\documentclass[12pt, dvipsnames]{amsart}
\usepackage{amssymb}
\usepackage{stmaryrd}
\usepackage[shortlabels]{enumitem}
\usepackage[all]{xy}
\usepackage{xcolor}
\usepackage{mathrsfs}
\usepackage[margin=1in]{geometry} 
\usepackage{hyperref}
\usepackage{eucal}
\hypersetup{bookmarksdepth=2}
\hypersetup{colorlinks=true}
\hypersetup{linkcolor=blue}
\hypersetup{citecolor=blue}
\hypersetup{urlcolor=blue}

\usepackage{mathtools}

\makeatletter
\@addtoreset{equation}{section}
\makeatother

\numberwithin{equation}{section}
\newtheorem{theorem}[equation]{Theorem}
\newtheorem*{corollary*}{Corollary}
\newtheorem{proposition}[equation]{Proposition}
\newtheorem{lemma}[equation]{Lemma}
\newtheorem{corollary}[equation]{Corollary}

\newtheorem{maintheorem}{Theorem}

\theoremstyle{definition}
\newtheorem{rmk}[equation]{Remark}
\newenvironment{remark}[1][]{\begin{rmk}[#1] \pushQED{\qed}}{\popQED \end{rmk}}
\newtheorem{eg}[equation]{Example}
\newenvironment{example}[1][]{\begin{eg}[#1] \pushQED{\qed}}{\popQED \end{eg}}
\newtheorem{defnaux}[equation]{Definition}
\newenvironment{definition}[1][]{\begin{defnaux}[#1]\pushQED{\qed}}{\popQED \end{defnaux}}

\newcommand{\bA}{\mathbf{A}}

\newcommand{\bN}{\mathbf{N}}

\newcommand{\cO}{\mathcal{O}}

\newcommand{\bS}{\mathbf{S}}

\newcommand{\bV}{\mathbf{V}}

\newcommand{\fa}{\mathfrak{a}}

\newcommand{\fb}{\mathfrak{b}}

\newcommand{\fc}{\mathfrak{c}}

\newcommand{\fm}{\mathfrak{m}}

\newcommand{\fn}{\mathfrak{n}}

\newcommand{\fp}{\mathfrak{p}}

\newcommand{\fq}{\mathfrak{q}}

\newcommand{\arxiv}[1]{\href{http://arxiv.org/abs/#1}{{\tiny\tt arXiv:#1}}}
\newcommand{\stacks}[1]{\cite[\href{http://stacks.math.columbia.edu/tag/#1}{Tag~#1}]{stacks}}
\newcommand{\DOI}[1]{\href{http://doi.org/#1}{\color{purple}{\tiny\tt DOI:#1}}}

\let\ul\underline
\let\ol\overline
\let\lbb\llbracket
\let\rbb\rrbracket
\renewcommand{\phi}{\varphi}

\DeclareMathOperator{\im}{im}

\DeclareMathOperator{\Spec}{Spec}
\DeclareMathOperator{\Spf}{Spf}

\DeclareMathOperator{\Sym}{Sym}
\DeclareMathOperator{\gr}{gr}
\DeclareMathOperator{\rank}{rank}
\DeclareMathOperator{\Tor}{Tor}
\DeclareMathOperator{\Sing}{Sing}
\newcommand{\id}{\mathrm{id}}
\renewcommand{\Vec}{\mathrm{Vec}}
\newcommand{\GL}{\mathbf{GL}}
\newcommand{\defn}[1]{\emph{#1}}
\newcommand{\sing}{\mathrm{sing}}
\newcommand{\ns}{\mathrm{nsing}}
\newcommand{\elem}{\mathrm{elem}}
\newcommand{\core}{\mathrm{core}}
\newcommand{\ulambda}{\ul{\lambda}}

\DeclareMathOperator{\embcodim}{embcodim}

\title{The singular locus of a $\GL$-variety}

\author{Christopher Heng Chiu}
\address{Mathematical Institute, University of Bern, Switzerland}
\email{\href{mailto:christopher.chiu@unibe.ch}{christopher.chiu@unibe.ch}}
\urladdr{\url{https://sites.google.com/view/christopher-heng-chiu/}}

\author{Alessandro Danelon}
\address{Department of Mathematics, University of Michigan, Ann Arbor, MI, USA}
\email{\href{mailto:adanelon@umich.edu}{adanelon@umich.edu}}
\urladdr{\url{https://public.websites.umich.edu/~adanelon/}}

\author{Andrew Snowden}
\address{Department of Mathematics, University of Michigan, Ann Arbor, MI, USA}
\email{\href{mailto:asnowden@umich.edu}{asnowden@umich.edu}}
\urladdr{\url{http://www-personal.umich.edu/~asnowden/}}

\thanks{AS was supported by NSF grant DMS-2301871. CC was supported by SNSF postdoc fellowship 217058.}

\date{May 28, 2026}

\begin{document}

\begin{abstract}
A $\GL$-variety is a (typically) infinite dimensional variety $X$ equipped with an action of the infinite general linear group. In recent work of the first two authors with Draisma, a candidate definition for the singular locus of $X$ was put forth. The approach there made use of auxiliary finite dimensional varieties associated to $X$. In this paper, we give a number of characterizations of the singular locus that are intrinsic to $X$. This work shows that the candidate definition is clearly correct, and helps clarify the geometric meaning of singular points.
\end{abstract}

\maketitle
\tableofcontents

\section{Introduction}

In the last several years, the theory of $\GL$-varieties has been developed \cite{polygeom, imgclosure, unirat,  charp, CDDEF, draisma, Danelon, tcares} and applied to problems in commutative algebra \cite{DLL, stillman}, geometric complexity theory \cite{imgclosure}, and the model theory of tensors \cite{homoten, homoten2}. Recently, \cite{CDD} proposed a definition of the singular locus of a $\GL$-variety. In this paper, we take a closer look at the proposed definition, and see that it has many of the characteristics one would like. We also prove a theorem about the flat locus of a module, which we expect to be generally useful.

\subsection{$\GL$-varieties}

We review the theory of $\GL$-varieties in \S \ref{s:gl}, but make a few brief comments here to set the stage for our results. We work over a field $k$ of characteristic~0. A \defn{$\GL$-algebra} is a commutative algebra in the category of polynomial representations of the infinite general linear group $\GL$, and a \defn{$\GL$-variety} is the spectrum of reduced finitely $\GL$-generated $\GL$-algebra. Most $\GL$-varieties of interest are infinite dimensional.

Polynomial representations are equivalent to polynomial functors. Thus if $A$ is a $\GL$-algebra then we can regard it as a polynomial functor and evaluate on a finite dimensional vector $V$ to obtain an algebra $A\{V\}$. Similarly, if $X=\Spec(A)$ is a $\GL$-variety then we put $X\{V\}=\Spec(A\{V\})$, which is an ordinary finite dimensional variety. The functorial assignment $V \mapsto X\{V\}$ is called a \defn{$\Vec$-variety} in the literature (see, e.g., \cite{CDD}). The categories of $\GL$-varieties and $\Vec$-varieties are equivalent, so there is no loss of information in passing between the two, at least in principle.

\subsection{The singular locus}

The notion of singularity is central to algebraic geometry. One would therefore like to have a notion of what it means for a point on a $\GL$-variety $X$ to be singular or non-singular. Unfortunately, the most elementary approach, using the rank of the Jacobian matrix, does not really make sense in the infinite dimensional setting.

In \cite{CDD}, a solution to this problem was devised by leveraging the associated finite dimensional varieties $X\{k^n\}$. Namely, the main theorem of that paper states that there is a unique closed $\GL$-subvariety $X_{\sing}$ such that $X_{\sing}\{k^n\}$ is the singular locus of $X\{k^n\}$ for all $n \gg 0$. We call $X_{\sing}$ the \defn{singular locus} of $X$, and refer to its points as \defn{singular}. Let us examine the simplest example.

\begin{example} \label{ex:quad}
Let $\bV$ be the standard representation of $\GL$, let $X$ be the space of all symmetric bilinear forms on $\bV$, and let $X_{\le r}$ be the rank $\le r$ locus in $X$. Then $X$ is a $\GL$-variety, and each $X_{\le r}$ is a closed $\GL$-subvariety. The space $X\{k^n\}$ is the space of symmetric bilinear forms on $k^n$, and $X_{\le r}\{k^n\}$ is the rank $\le r$ locus. It is well-known that the singular locus of $X_{\le r}\{k^n\}$ is $X_{\le r-1}\{k^n\}$ for $n>r$, and so $(X_{\le r})_{\sing}$ is $X_{\le r-1}$.
\end{example}

While the approach to the singular locus in \cite{CDD} is sensible, it raises some obvious questions. We mention two.

(1) A $\GL$-variety is a scheme with an action of $\GL$. To see the $X\{k^n\}$ spaces, one must use the action of $\GL$. Thus the definition of $X_{\sing}$ depends not just on the underlying scheme $X$, but also on the action. A natural question is if one can detect singular points intrinsically from the scheme. More specifically, one can ask which, if any, of the classical tests for singularity (such as formal smoothness, or local freeness of differentials) apply in the $\GL$-variety case.

(2) The $\Vec$-variety only sees the ``finite'' points of a $\GL$-variety. For example, the $X$ in Example~\ref{ex:quad} has points of infinite rank, but every point in $X\{k^n\}$ has finite rank. Thus, with the present approach, the geometric significance of singular is quite unclear for ``infinite'' points. We note that these infinite points are the most interesting ones in the model-theoretic applications \cite{homoten, homoten2}, so it is important to understand them well.

\subsection{Main results}

The following is our main theorem. It provides some answers to the questions discussed above, and, in turn, some justification for the definition of $X_{\sing}$. We emphasize that, for finite-dimensional varieties over a field of characteristic $0$, (b), (c), (d), and (g) are equivalent to the usual property of being regular. In this way, Theorem~\ref{mainthm} shows that $\GL$-varieties do not exhibit pathologies that appear in the infinite-dimensional setting more generally (see \S \ref{ss:related} below).

\begin{maintheorem} \label{mainthm}
Let $X$ be a $\GL$-variety and let $x$ be a point of $X$. The following conditions are equivalent:
\begin{enumerate}
\item The point $x$ is non-singular, i.e., $x \not\in X_{\sing}$.
\item The local ring $\cO_{X,x}$ is regular (Definition~\ref{defn:regular}): the tangent space and tangent cone at $x$ coincide.
\item The jet lifting property holds at $x$ (Definition~\ref{defn:jlp}): every $m$-jet at $x$ lifts to an $(m+1)$-jet, for every $m \ge 0$.
\item The local ring $\cO_{X,x}$ is a formally smooth $k$-algebra (Definition~\ref{defn:fsm}).
\item The local ring $\cO_{X,x}$ is a localization of a ring of the form $A[t_i]_{i \in I}$, where $A$ is a smooth (and thus finitely generated) $k$-algebra and $I$ is an index set.
\item There is non-singular open elementary neighborhood of $x$ (see \S \ref{ss:shift}).
\item The module $\Omega_X$ of K\"ahler differentials is flat at $x$.
\end{enumerate}
\end{maintheorem}

We note that (b)--(e) and (g) only depend on the underlying scheme and not the $\GL$-action. To prove the equivalence of (g) with the other conditions, we establish a general theorem about the flat locus of a module, which we expect to be generally useful in the study of $\GL$-varieties. This is our result:

\begin{maintheorem} \label{mainthm3}
Let $A$ be a reduced $\GL$-algebra that is finitely $\GL$-generated, let $X=\Spec(A)$, and let $M$ be a finitely $\GL$-generated $A$-module. Then:
\begin{enumerate}
\item The locus $U$ in $X$ where $M$ is flat is an open dense $\GL$-stable subset.
\item The flat locus is compatible with evaluation in the following sense. Let $Z=X \setminus U$. Then, for $n$ sufficiently large, $X\{k^n\} \setminus Z\{k^n\}$ is the locus where $M\{k^n\}$ is flat.
\end{enumerate}
\end{maintheorem}

Applying this theorem to the module of differentials gives the equivalence of (a) and (g) in Theorem~\ref{mainthm}. In fact, this gives a new proof of the main result of \cite{CDD} (see Remark~\ref{rmk:newproof}).

In the course of proving Theorem~\ref{mainthm}, we introduce the notion of the core of a $\GL$-variety $X$. The shift theorem of \cite{polygeom} shows that $X$ has a dense open elementary subvariety (see \S \ref{ss:shift} for the definition). We define the \defn{core} of $X$ to be the locus of points that do not belong to any open elementary subvariety; these are essentially the points where the local structure of $X$ is complicated. The core is a closed $\GL$-subvariety of $X$. We give an alternative characterization of the core using the notion of embedding codimension of \cite{CdFD}; this is an important step in our proof of Theorem~\ref{mainthm}:

\begin{maintheorem} \label{mainthm4}
Let $X$ be a $\GL$-variety. Then $x \in X$ belongs to the core of $X$ if and only if the embedding codimension of $\cO_{X,x}$ is infinite, that is, the tangent cone at $x$ has infinite codimension in the tangent space at $x$.
\end{maintheorem}

The theorem shows that the core is intimately related to the study of singularities. It also implies that the core can be recovered directly from $X$ as a scheme, without use of the $\GL$-action. We believe this could be a first step towards a finer understanding of the local geometry of $\GL$-varieties at singular points.

\subsection{Related work}
\label{ss:related}

Singularities of infinite-dimensional schemes have been studied as part of \cite{CdFD} and \cite{CdFD2} in the context of arc spaces of algebraic varieties. 
One of the main results in \cite{CdFD} characterizes those arcs for which a formal-local decomposition \`a la Drinfeld exists (see \cite{Chi}) in terms of the embedding codimension of their local rings being finite. In contrast, Theorem~\ref{mainthm4} shows that finiteness of embedding codimension forces the local structure of a $\GL$-variety to be much more rigid: in our setting, a similar decomposition is already necessarily Zariski-local.

Of note is that singularities of arc spaces seem to exhibit more pathological behavior than those of $\GL$-varieties. For example, any arc contained in a smooth formal branch is regular \cite{BS}. However, such arcs are in general only formally smooth if their center is a smooth point (i.e., there is only one formal branch); and similarly the Krull intersection theorem only holds then \cite{BH}. We give a detailed explanation in Example~\ref{ex:not-form-smooth-2} here. 

\subsection{Notation}

We list some important notation:
\begin{description}[align=right,labelwidth=1.9cm,leftmargin=!]
\item[ $k$ ] the coefficient field (characteristic~0) 
\item[ $\GL$ ] the infinite general linear group (\S \ref{ss:polyrep})
\item[ $G(n)$ ] a block subgroup of $\GL$ (\S \ref{ss:polyrep})
\item[ $\cO_{X,x}$ ] the local ring of a scheme $X$ at a point $x$
\item[ $\kappa(x)$ ] the residue field of $\cO_{X,x}$
\item[ $\Omega_X$ ] the K\"ahler differentials of $X$ over $k$
\item[ $\bA(V)$ ] the affine space associated to the polynomial representation $V$ (\S \ref{ss:glvar})
\item[ $X\{V\}$ ] the evaluation of a $\GL$-variety $X$ on a vector space $V$ (\S \ref{ss:glvar})
\end{description}

\subsection*{Acknowledgments}

We thank Jan Draisma and Mircea Mustata for helpful discussions. We are also grateful to Alapan Mukhopadhyay for pointing us towards Example~\ref{ex:not-form-smooth-3}.

\section{Background on $\GL$-varieties} \label{s:gl}

In this section, we review some background material on polynomial representations, $\GL$-varieties, and related topics. We refer to \cite{polygeom} for additional details.

\subsection{Polynomial representations} \label{ss:polyrep}

Fix a field $k$ of characteristic~0. Let $\GL$ be the union of the groups $\GL_n(k)$, with respect to the standard inclusions. Let $\bV=\bigcup_{n \ge 1} k^n$ be the standard representation of $\GL$, with basis $\{e_i\}_{i \ge 1}$. A representation of $\GL$ is \defn{polynomial} if it is a subquotient of a direct sum of tensor powers of $\bV$. The category of such representations is semi-simple, with the simple objects being the Schur functors $\bS_{\lambda}(\bV)$, where $\lambda$ is a partition. It is also closed under the tensor product. Polynomial representations are naturally graded, with $\bS_{\lambda}(\bV)$ in degree $\vert \lambda \vert$. In particular, the degree~0 piece of a polynomial representation coincides with its $\GL$-invariants.

The category of polynomial representations of $\GL$ is equivalent to the category of polynomial functors, with the representation $\bS_{\lambda}(\bV)$ corresponding to the functor $\bS_{\lambda}$. Given a polynomial representation $A$ and a vector space $V$, we write $A\{V\}$ for the value of the functor corresponding to $A$ on $V$. We will typically apply this when $V=k^n$. For $V$ fixed, the functor $A \mapsto A\{V\}$ is exact and compatible with tensor products.

Let $G(n)$ be the subgroup of $G$ consisting of block matrices of the form
\begin{displaymath}
\begin{pmatrix} 1 & 0 \\ 0 & \ast \end{pmatrix}
\end{displaymath}
where the top left block has size $n \times n$. The group $G(n)$ is isomorphic to $\GL$, so any definition or construction we make for $\GL$ also applies to $G(n)$. In particular, we can speak of polynomial representations of $G(n)$. If $A$ is a polynomial representation of $\GL$ then $A\{k^n\}$ is naturally identified with the invariant space $A^{G(n)}$. In particular, the functor $(-)^{G(n)}$ is exact and compatible with tensor products on the category of polynomial representations.

Suppose $A$ is a polynomial representation of $\GL$. Restricting the action to $G(n)$, we obtain a polynomial representation of $G(n)$. We can then identify $G(n)$ with $\GL$ and regard $A$ as a representation of $\GL$ again; in other words, we are simply restricting $A$ along the standard self-embedding $\GL \to \GL$ with image $G(n)$. We refer to this operation as \defn{shifting}. From the polynomial functor point of view, the shift of $A$ is the functor $V \mapsto A\{V \oplus k^n\}$. In \cite{polygeom}, the shift operation is used extensively. In this paper, we tend to simply restrict to $G(n)$, and not identify this group explicitly with $\GL$; we still refer to this as shifting.

\subsection{$\GL$-algebras}

A \defn{$\GL$-algebra} is a commutative algebra object $A$ in the category of polynomial representations. Explicitly, $A$ is an associative, commutative, and unital $k$-algebra equipped with an action of $\GL$ under which it forms a polynomial representation. We say that $A$ is \defn{finitely $\GL$-generated} if it is generated as a $k$-algebra by the $\GL$-orbits of finitely many elements; equivalently, this mean that $A$ is a quotient of $\Sym(V)$ for some finite length polynomial representation $V$.

Let $A$ be a $\GL$-algebra. By an \defn{$A$-module} we mean a module for $A$ internal to the category of polynomial representations. Suppose $M$ is an $A$-module. We say that $M$ is \defn{finitely $\GL$-generated} if it is generated as an $A$-module by the $\GL$-orbits of finitely many elements. Equivalently, this means that $M$ is a quotient of an $A$-module of the form $A \otimes V$ for some finite length polynomial representation $V$. In particular, if $\fa$ is a $\GL$-stable ideal of $A$ then it is an $A$-module in this sense, and we can speak of finite $\GL$-generation.

\subsection{$\GL$-varieties} \label{ss:glvar}

A \defn{$\GL$-variety} is a scheme $X$ of the form $\Spec(A)$, where $A$ is a reduced finitely $\GL$-generated $\GL$-algebra. 
While one could consider generalizations to allow for non-affine schemes, a good notion of regularity should be local for the Zariski topology and thus it will be sufficient to restrict to the affine case here.
For a vector space $V$, we write $X\{V\}$ for the scheme $\Spec(A\{V\})$. When $V$ is finite dimensional, $X\{V\}$ is an ordinary finite dimensional algebraic variety.

Suppose $V$ is a finite length polynomial representation of $\GL$. Then $\Sym(V)$ is a reduced and finitely $\GL$-generated $\GL$-algebra, and so $\bA(V)=\Spec(\Sym(V))$ is a $\GL$-variety. These play a role in the theory analogous to the affine spaces $\bA^n$ in classical algebraic geometry. By definition, every $\GL$-variety is a closed $\GL$-subvariety of some $\bA(V)$. The notation $\bA(V)$ was used in \cite{charp}, but in \cite{polygeom} the notation $\bA^{\ulambda}$ was employed.

Let $X=\Spec(A)$ be a $\GL$-variety. For an element $h \in A$, we write $X[1/h]$ for the open set where $h$ is non-zero, which we identify with the affine scheme $\Spec(A[1/h])$. If $h$ is invariant under $G(n)$ then $X[1/h]$ is an open $G(n)$-subvariety of $X$.

Let $X$ be a $\GL$-variety and let $x,y \in X$ be two points. We say that $x$ and $y$ belong to the same \defn{generalized $\GL$-orbit} if each belongs to the Zariski closure of the orbit of the other. This is an equivalence relation on $X$. The generalized orbit of $x$ contains a unique $\GL$-fixed point, namely, the generic point of the closure of the orbit of $x$. A locally closed $\GL$-stable set is a union of generalized orbits. See \cite[\S 3]{polygeom} for these results, and more.

\begin{remark}
Let $X$ be a $\GL$-variety. Then $V \mapsto X\{V\}$ defines a contravariant functor from the category of finite dimensional vector spaces to the category of algebraic varieties. Such a functor (with appropriate conditions) is called a \defn{$\Vec$-variety} in the literature. The notions of $\GL$-variety and $\Vec$-variety are equivalent. We note that \cite{CDD} uses the $\Vec$-variety point of view.
\end{remark}

\subsection{The singular locus}

Let $X$ be a $\GL$-variety. As explained in the introduction, the main theorem of \cite{CDD} asserts that there is a unique closed $\GL$-subvariety $X_{\sing}$ of $X$ such that $X_{\sing}(k^n)$ is the singular locus in $X\{k^n\}$ for all $n \gg 0$. We call $X_{\sing}$ the \defn{singular locus} of $X$ and we refer to its points as \defn{singular}. We let $X_{\ns}$ be the complement of $X_{\sing}$, which is a dense open $\GL$-stable subset of $X$; its points are called \defn{non-singular}.

Let $X=\Spec(A)$. Identifying $A\{k^n\}$ with the $G(n)$-invariants of $A$, we see that there is a corresponding map $\pi_n \colon X \to X\{k^n\}$ of affine schemes. Since $A$ is the union of the $A\{k^n\}$, it follows that $X$ is the inverse limit of $X\{k^n\}$; this holds both as affine schemes, and also just as sets (forgetting all the extra structure). Now suppose that $Z$ is a closed $\GL$-subvariety of $X$. It then follows that a point $x \in X$ belongs to $Z$ if and only if $\pi_n(x)$ belongs to $Z\{k^n\}$ for all $n$, or, equivalently, for all sufficiently large $n$. In particular, taking $Z=X_{\sing}$, we see that $x$ is a singular point of $X$ if and only if $\pi_n(x)$ is a singular point of $X\{k^n\}$ for all $n \gg 0$. We note that this applies to all scheme-theoretic points of $X$, not just closed points.

\subsection{Generic freeness}

We will require the following theorem on modules over $\GL$-algebras. It is proved in \cite[Theorem~3.3]{tcares}. We note that this is a module-theoretic analog of the shift theorem proved in \cite{polygeom}, and which appears in \S \ref{ss:shift} below.

\begin{theorem} \label{thm:genfree}
Let $A$ be an integral $\GL$-algebra and let $M$ be a finitely $\GL$-generated $A$-module. Then there exists $n \ge 0$, a non-zero $G(n)$-invariant element $h \in A$, and a finite length polynomial representation $V$ of $G(n)$ such that $M[1/h]$ is isomorphic to $V \otimes A[1/h]$ as a $G(n)$-equivariant $A[1/h]$-module.
\end{theorem}

%
%

\section{Some commutative algebra}

In this section, we review a few different notions of regularity and smoothness for commutative rings. Throughout, we work over a field $k$ of characteristic~0; many results can be extended to positive characteristic by including separability hypotheses. We note this section has nothing to do with $\GL$-varieties.

\subsection{Formal smoothness}

Let $A$ be a ring and let $\fa$ be an ideal of $A$. Recall that the \defn{$\fa$-adic} topology on $A$ is the topology in which the powers of $\fa$ form a neighborhood basis of the identity. If $\fa=0$ then the $\fa$-adic topology is the discrete topology. If $A$ is equipped with the $\fa$-adic topology and $B$ is equipped with the $\fb$-adic topology then a ring homomorphism $\phi \colon A \to B$ is continuous if and only if $\phi(\fa^n) \subset \fb$ for some $n$. In particular, if $B$ has the discrete topology then $\phi$ is continuous if and only if it factors through $A/\fa^n$ for some $n$.

\begin{definition} \label{defn:tfsm}
Let $\phi \colon A \to B$ be a continuous ring homomorphism with respect to adic topologies. We say that $\phi$ is \defn{topologically formally smooth} if the following condition holds. Suppose we are given the outside commutative square in the following diagram
\begin{displaymath}
\xymatrix@C=4em{
B \ar[r]^{\beta} \ar@{..>}[rd]^-{\gamma} & R/I \\
A \ar[u]^{\phi} \ar[r]^-{\alpha} & R \ar[u] }
\end{displaymath}
Here $R$ is any ring, $I$ is a square-zero ideal of $R$, and $\alpha$ and $\beta$ are continuous ring homomorphisms, where $R$ and $R/I$ have the discrete topology. Then we require a continuous ring homomorphism $\gamma$ to exist that makes the diagram commute.
\end{definition}

\begin{definition} \label{defn:fsm}
A ring homomorphism $\phi \colon A \to B$ is \defn{formally smooth} if it is topologically formally smooth when $A$ and $B$ are given the discrete topologies.
\end{definition}

We will mostly be interested in these definitions when $B$ is a $k$-algebra and $A=k$. If $\phi \colon A \to B$ is a continuous ring homomorphism that is formally smooth then it is necessarily topologically formally smooth \stacks{07EC}. The converse is not true in general, as the following examples show.

\begin{example} \label{ex:not-form-smooth-1}
    Let $k$ be a field of characteristic $\neq 2$. For $n\geq 0$, set $A_n := k[x_n]_{(x_n)}$ and consider the map $A_n \to A_{n+1}$ given by $x_n \mapsto x^2_{n+1}$. Take $A := \varinjlim_n A_n$ to be the colimit of this direct system. Then $A$ is given by taking the quotient of the infinite-variate polynomial ring $k[x_0,x_1,\ldots]$ by the ideal generated by $x_n - x^2_{n+1}$ for all $n\geq0$ and localizing at the prime ideal $\fm$ generated by all the variables. Clearly $x_n \in \bigcap_n \fm$, and thus the $\fm$-adic completion $\widehat A$ of $A$ is isomorphic to $k$. In particular, $A$ is topologically formally smooth over $k$. On the other hand, it is easy to check that $\Omega_{A/k} \neq 0$ but $\Omega_{A/k} \otimes_A \widehat A = 0$, thus showing that $\Omega_{A/k}$ is not projective and hence $A$ not formally smooth over $k$ \stacks{031J}. Of note is that $A$ in this case is a directed colimit of smooth rings with non-smooth transition maps. We will show in Lemma~\ref{lem:colimit-regular} (together with the main result in Proposition~\ref{prop:compare}) that any such colimit is topologically formally smooth.
\end{example}

\begin{example} \label{ex:not-form-smooth-2}
    Such pathologies are common when looking at local rings of the arc space of an algebraic variety. We refer the reader to \cite{CdFD2} for details, but include a concrete example here as it is not explicitly mentioned in the literature. Let $k$ be a field of characteristic $0$ and $X = \Spec k[x,y]/(xy)$. Then the arc space $X_\infty$ is an affine subscheme of $\Spec k[x_0,x_1,\ldots,y_0,y_1,\ldots]$. Let $\alpha \in X_\infty$ be an arc of the form $(\alpha(t),0) \in k_\alpha \lbb t \rbb^2$, with $\alpha(t) = \sum \alpha_i t^i$ and $k_\alpha$ a field extension of $k$. Geometrically this means that $\alpha$ is contained in the irreducible component defined by $y=0$. Let $g = x_0 y_1 \in \mathcal{O}_{X_\infty,\alpha}$, which satisfies $g^2 = 0$. Similar to \cite[Example 5.12]{CdFD2}, one can verify that $dg \neq 0$ in $\Omega^1_{X_\infty,\alpha}$ and hence, again by \stacks{031J}, the ring $\mathcal{O}_{X_\infty,\alpha}$ is not formally smooth. To be precise, let $R \subset k_\alpha(r,s)$ be the valuation ring induced by the monomial valuation $\nu \colon k_\alpha(r,s)^* \to \mathbb{Z}^2$ given by $\nu(s) = (1,0)$, $\nu(r) = (0,1)$. Let $\widetilde R = R/(rs)$ and $\varphi \colon \mathcal{O}_{X_\infty, \alpha} \to \widetilde R$ be the map induced by
    \[
        x \mapsto r - \alpha(t),\quad y \mapsto s + \frac{s}{r} \alpha(t) + \frac{s}{r^2} \alpha^2(t) + \ldots
    \]
    Then $\varphi(g) = s \alpha_1$. Thus it suffices to show that $ds$ is not zero in $\Omega_{\widetilde R/k_\alpha}$. To that avail, take the $k_\alpha$-derivation $\partial \colon k_\alpha(r,s)$ defined by $\partial(r) = -r$ and $\partial(s) = s$. It is immediate to check that $\partial$ induces a derivation on $\widetilde R = R/(rs)$ with $\partial(s) \neq 0$, which proves the claim. 
    
    On the other hand, the ring $\mathcal{O}_{X_\infty,\alpha}$ is topologically formally smooth by \cite[Theorem 1.6]{BS} and, as in Example~\ref{ex:not-form-smooth-1}, it is not separated for the adic topology by \cite[Theorem 1.1]{BH}.
\end{example}

\begin{example} \label{ex:not-form-smooth-3}
For a simpler example of a different nature consider the ring of formal series $k\lbb t \rbb$ over a field $k$, which is separated for the $t$-adic topology. Clearly $k\lbb t \rbb$ is topologically formally smooth, but it is only formally smooth if the characteristic of $k$ is $p > 0$ and the degree of the extension $k^p/k$ is finite, see \cite[Lemma 2.1]{Tan}.
\end{example}

\subsection{Jets}

Let $A$ be a $k$-algebra. An \defn{$m$-jet} valued in a field extension $K/k$ is a $k$-algebra homomorphism $\alpha \colon A \to K[t]/(t^{m+1})$. Given an $m$-jet $\alpha$, let $\alpha_0 \colon A \to K$ be the composition of $\alpha$ with the reduction mod $t$ map. Then $\ker(\alpha_0)$ is a prime ideal $\fp$ of $A$, called the \defn{center} of $\alpha$. Geometrically, an $m$-jet is a map of $k$-schemes $\Spec(K[t]/(t^{m+1})) \to \Spec(A)$, and the center is simply the image of the unique point in the source. Observe that there is a bijection between jets of $A$ centered at $\fp$ and jets of the localization $A_{\fp}$ centered at $\fp A_{\fp}$.

\begin{definition} \label{defn:jlp}
Let $A$ be a $k$-algebra and let $\fp$ be a prime ideal of $A$. We say that $A$ has the \defn{jet lifting property} at $\fp$ if for every $m$, any $m$-jet of $A$ valued in $K/k$ and centered at $\fp$ lifts to an $(m+1)$-jet of $A$ valued in $K$ (which is necessarily centered at $\fp$).
\end{definition}

From the above comments, we see that $A$ has the jet lifting property at $\fp$ if and only if $A_{\fp}$ has the property at $\fp A_{\fp}$. The jet lifting property is essentially a special case of the property appearing in the definition of formal smoothness, and so we have the following simple implication:

\begin{proposition} \label{prop:jlp-tfsm}
Let $A$ be a $k$-algebra and let $\fp$ be a prime ideal of $A$. If $A$, equipped with the $\fp$-adic topology, is topologically formally smooth over $k$ then $A$ has the jet lifting property at $\fp$.
\end{proposition}

\begin{proof}
Let $\alpha \colon A \to K[t]/(t^{m+1})$ be a jet centered at $\fp$. Note that $\fp^{m+1} \subset \ker(\alpha)$, which shows that $\alpha$ is continuous, where we use the $\fp$-adic topology on the source and the discrete topology on the target. Consider the following diagram:
\begin{displaymath}
\xymatrix@C=4em{
A \ar[r]^-{\alpha} \ar@{..>}[rd]_-{\beta} & K[t]/(t^{m+1}) \\
k \ar[u] \ar[r] & K[t]/(t^{m+2}) \ar[u] }
\end{displaymath}
Since $\alpha$ is continuous, topological formal smoothness implies that the map $\beta$ exists. We have thus lifted $\alpha$ to an $(m+1)$-jet, as required.
\end{proof}

In the jet lifting property, we require that the lift of our jet be valued in the same field. In fact, allowing for extension fields yields an equivalent condition:

\begin{proposition} \label{prop:jlp-field}
Let $A$ be a $k$-algebra, let $\fp$ be a prime ideal, and let $\alpha$ be an $m$-jet of $A$ centered at $\fp$ and valued in $K/k$. Suppose that $\alpha$ lifts to an $(m+1)$-jet over some extension field of $K$. Then $\alpha$ also lifts to an $(m+1)$-jet over $K$.
\end{proposition}

\begin{proof}
Choose a presentation $A=S/I$, where $S=k[x_i]_{i \in I}$ is a polynomial ring and $I=(f_j)_{j \in J}$ is an ideal of $S$. Let $\alpha_i=\alpha(x_i)$, which we represent as an element of $K[t]$ of degree $\le m$. We now attempt to construct a lift $\beta$ of $\alpha$ to an $(m+1)$-jet valued in an extension $E/K$. Let $\{c_i\}_{i \in I}$ be elements of $E$, and define a ring homomorphism $\beta \colon S \to E[t]/(t^{m+2})$ by $\beta(x_i)=\alpha_i+c_i t^{m+1}$. The homomorphism $\beta$ factors through $I$ if and only if $\beta(f_j)=0$ for all $j \in J$. We have
\begin{displaymath}
\beta(f_j) = f_j(\alpha_i+c_i t^{m+1}) = \ell_j(c_i) t^{m+1},
\end{displaymath}
where $\ell_j$ is a (possibly inhomogeneous) linear form with coefficients in $K$, and we have discarded higher degree terms (since we are working modulo $t^{m+2}$). We thus see that $\beta$ factors through $A$ if and only if $\ell_j(c_i)=0$ for all $j \in J$.

Now, by assumption, there is a lift of $\alpha$ over some $E$. This means the system of equations $\ell_j=0$ has a solution in $E$. Since this is a system of linear equations, it follows that it has a solution in $K$. This gives a lift valued in $K$, as required.
\end{proof}

The jet lifting property can be formulated in terms of jet spaces, as we now explain. Let $X$ be a $k$-scheme. Recall that the \defn{$m$-th jet space}, denoted $J_m(X)$, is the $k$-scheme that represents the functor that maps a $k$-algebra $R$ to $X(R[t]/(t^{m+1}))$, the set of $(m+1)$-jets of $X$ valued in $R$. It is well-known that $J_m(X)$ is representable; see \cite{EM} for background. If $X$ is of finite type over $k$ then so is $J_m(X)$, and the natural map $J_m(X) \to X$ is affine. If $X=\Spec(A)$ is affine the $J_m(X)$ can be identified with the spectrum of $\operatorname{HS}^m_{A/k}$, the universal algebra of order $m$ higher derivations; see \cite{Voj}. For $x \in X$, we let $J_m(X)_x$ denote the fiber of $J_m(X) \to X$ over $x$.

\begin{proposition}
Let $X$ be a $k$-scheme and let $x \in X$. Then $\cO_{X,x}$ has the jet lifting property if and only if for every $m \ge 0$ the map $J_{m+1}(X)_x \to J_m(X)_x$ is surjective.
\end{proposition}

\begin{proof}
This follows from Proposition~\ref{prop:jlp-field}.
\end{proof}

\subsection{Regularity}

Recall that a Noetherian local ring is regular if its embedding dimension and dimension agree. Hence one may think of their difference as a measure of singularities. In \cite{CdFD} this notion of embedding codimension was extended to non-Noetherian rings, which we will recall here. For any local ring $(A,\fm)$ we consider the natural surjection
\begin{equation} \label{eq:nat-surj}
\gamma \colon \Sym(\fm/\fm^2) \to \gr(R) = \bigoplus_{n \ge 0} \fm^n/\fm^{n+1}.
\end{equation}
where $\Sym$ denotes the symmetric algebra in the category of $A/\fm$-vector spaces. Then the \defn{embedding codimension} of $(A,\fm)$ is defined as
\[
	\embcodim(A) \coloneqq \operatorname{ht}(\ker \gamma).
\]
This is either a natural number or $\infty$. Geometrically the embedding codimension measures the codimension of the tangent cone inside the tangent space.

Motivated by the above we adopt the following notion of regularity for non-Noetherian local rings:

\begin{definition} \label{defn:regular}
A local ring $(A, \fm)$ is \defn{regular} if the natural map
\begin{displaymath}
\Sym(\fm/\fm^2) \to \bigoplus_{n \ge 0} \fm^n/\fm^{n+1}
\end{displaymath}
is an isomorphism, i.e., the embedding codimension is~0. A ring $A$ is \defn{regular} at a prime ideal $\fp$ if the local ring $A_{\fp}$ is regular, and a scheme $X$ is \defn{regular} at a point $x$ if the local ring $\cO_{X,x}$ is regular.
\end{definition}

Geometrically then, $(A, \fm)$ is regular if the tangent cone and tangent space coincide at the closed point of $\Spec{A}$. Sometimes (lack of) regularity can be transferred along a map of schemes. We recall the following result from \cite{CdFD}.

\begin{proposition}[{\cite[Proposition 6.6]{CdFD}}]
\label{prop:embcodim-map}
Let $\varphi \colon (B,\fn) \to (A,\fm)$ be a homomorphism of local rings and write $K = A/\fm$ and $L = B/\fn$. Assume that $B$ is Noetherian. Let $\varphi^* \colon \fn/\fn^2 \otimes_L K \to \fm/\fm^2$ be the induced $K$-linear map on Zariski cotangent spaces. Then:
\[
	\embcodim(A) \geq \rank(\varphi^*) - \dim(B).
\]
\end{proposition}

In Proposition~\ref{prop:mainthm-3} we will make use of a special case of Proposition~\ref{prop:embcodim-map}. Since the argument simplifies (and does not require any Noetherianity hypotheses) we include it for the reader's convenience. First, we start with the following definition.

\begin{definition} \label{defn:subm}
Let $f \colon Y \to X$ be a map of schemes, let $y \in Y$, and let $x=f(x)$. We say that $f$ is \defn{submersive} at $y$ if the natural map
\begin{displaymath}
\kappa(y) \otimes_{\kappa(x)} \fm_x/\fm_x^2 \to \fm_y/\fm_y^2
\end{displaymath}
is injective.
\end{definition}

\begin{lemma} \label{prop:reg-contract}
Let $f \colon Y \to X$ be a map of schemes, let $y \in Y$, and let $x=f(y)$. If $f$ is submersive at $y$ and $Y$ is regular at $y$ then $X$ is regular at $x$.
\end{lemma}

\begin{proof}
By assumption we have the following diagram
\[
\xymatrix{ \Sym(\fm_x/\fm_x^2) \ar[r] \ar@{->>}[d]_{\gamma_x} & \Sym(\fm_y/\fm_y^2) \ar@{->>}[d]^{\gamma_y} \\
	\gr(\cO_{X,x}) \ar[r] & \gr(\cO_{Y,y})}
\]
with the top horizontal and the right vertical maps being injective. Hence $\gamma_x$ is injective, which implies that $X$ is regular at $x$.
\end{proof}


We also want to mention the following fact, which we have already alluded to previously.

\begin{lemma}
\label{lem:colimit-regular}
Let $(A,\fm)$ be a local ring that is the directed colimit of local rings $(A_n,\fm_n)$. If $(A_n,\fm_n)$ is regular for $n \gg 0$, then $(A,\fm)$ is regular.
\end{lemma}

\begin{proof}
It is straightforward to check that the natural map \eqref{eq:nat-surj} is the directed colimit of the natural maps
\[
	\Sym(\fm_n/\fm_n^2) \to \gr(R_n),
\]
which by assumption are all isomorphism for $n \gg 0$. Hence \eqref{eq:nat-surj} is an isomorphism.
\end{proof}

\subsection{Comparison}

The following proposition compares the various notions of smoothness and regularity discussed above.

\begin{proposition} \label{prop:compare}
Let $(A, \fm)$ be a local $k$-algebra. The following are equivalent:
\begin{enumerate}
\item $A$ is regular.
\item $A$ is topologically formally smooth over $k$ in the $\fm$-adic topology.
\item $A$ has the jet lifting property at $\fm$.
\end{enumerate}
These conditions are implied by the following:
\begin{enumerate}
\setcounter{enumi}{3}
\item $A$ is formally smooth over $k$.
\end{enumerate}
There exist local $k$-algebras that satisfy (a)--(c) and not (d).
\end{proposition}

\begin{proof}
The equivalence of (a) and (b) is \cite[Chapter 0, 19.5.4]{EGAIVa}. We have seen (Proposition~\ref{prop:jlp-tfsm}) that (b) implies (c). We have also remarked that (d) implies (b) \stacks{07EC}, and have given examples (Examples~\ref{ex:not-form-smooth-1} and~\ref{ex:not-form-smooth-2}) to show that the converse can fail.

To complete the proof, we show that (c) implies (a); we actually prove the contrapositive. Thus suppose $A$ is not regular. Let $\kappa=A/\fm$. Since $A$ is not regular, we can find $\kappa$-linearly independent elements $y_1, \ldots, y_n$ of $\fm/\fm^2$ and a homogeneous polynomial $f \in \kappa[T_1, \ldots, T_n]$ of some degree $d \ge 2$ such that $f(y_1, \ldots, y_n)=0$ as an element of $\fm^d/\fm^{d+1}$. Let $x_i \in \fm$ be a lift of $y_i$ and let $F \in A[T_1, \ldots, T_n]$ be a lift of $f$. Thus $F(x_1, \ldots, x_n)$ belongs to $\fm^{d+1}$.

Since $\kappa$ is infinite, we can find a point $(a_1, \ldots, a_n) \in \kappa^n$ such that $f(a_1, \ldots, a_n) \ne 0$; fix such a choice. Also, since we are in characteristic~0, $\kappa$ is formally smooth over $k$ \stacks{0322}, which implies that the natural map $A/\fm^2 \to \kappa$ splits as a map of $k$-algebras. Fix a choice of splitting, and thus regard $A/\fm^2$ as a $\kappa$-algebra. Giving a $\kappa$-linear map $\alpha \colon A/\fm^2 \to \kappa[t]/(t^2)$ is equivalent to giving a $\kappa$-linear map $\alpha_0 \colon \fm/\fm^2 \to \kappa$. Let $\alpha_0$ be any linear map such that $\alpha_0(y_i)=a_i$ for $1 \le i \le n$, which exists since the $y_i$ are linearly independent, and let $\alpha$ be the corresponding map; note that $\alpha(x_i)=a_it$. We regard $\alpha$ as a 1-jet of $A$ centered at $\fm$. We claim that $\alpha$ does not lift to a $d$-jet, which will complete the proof.

Suppose $\beta \colon A \to \kappa[t]/(t^{d+1})$ is a $d$-jet lifting $\alpha$, meaning that $\beta$ is a map of $k$-algebras that agrees with $\alpha$ modulo $t^2$. Since $\beta(\fm) \subset (t)$, we have $\beta(\fm^{d+1})=0$, and so $\beta(F(x_1, \ldots, x_n))=0$. We now compute this in a different way. Write $F(T_1, \ldots, T_n) = \sum_{e} C_e T^e$, where the sum is over multi-indices $e$ of total degree $d$, and $C_e \in A$ is a coefficient; let $c_e \in \kappa$ be the image of $C_e$. Since $\beta$ agrees with $\alpha$ modulo $t^2$ and $\alpha$ is $\kappa$-linear, we have
\begin{displaymath}
\beta(C_e) = \alpha(C_e)+O(t^2) = c_e+O(t), \qquad
\beta(x_i) = \alpha(x_i)+O(t^2) = a_i t + O(t^2).
\end{displaymath}
We thus see that $\beta(C_e x^e) = c_e a^e t^d + O(t^{d+1})$, and the error term here does not matter since we are working modulo $(t^{d+1})$. We therefore have
\begin{displaymath}
\beta(F(x_1, \ldots, x_n)) = f(a_1, \ldots, a_n) t^d \ne 0,
\end{displaymath}
which is a contradiction. This completes the proof.
\end{proof}

\begin{remark} \label{rmk:1-jet}
The proof actually shows a bit more: if every 1-jet of $A$ centered at $\fm$ lifts to a $d$-jet, for every $d \ge 2$, then $A$ is regular, and thus has the jet lifting property at $\fm$.
\end{remark}

%
%

\section{The core of a $\GL$-variety} \label{s:core}

In this section, we introduce the notion of the core of a $\GL$-variety. The core is the locus of points where the variety is not locally elementary, i.e., a product of a finite-dimensional variety and an affine space. Our main result is that the core can be characterized in terms of embedding codimension; we note that the key idea here is due to \cite{CDD}. In particular, the core is always contained in the singular locus. 
This means that the local structure at a non-singular point is easy (see Corollary~\ref{cor:core-sing}), and this is a key input into the proof of Theorem~\ref{mainthm}. We believe that the core could be a useful concept in the future study of $\GL$-varieties.

\subsection{Magnitude}

Given a polynomial representation $A$, we write $A_n$ for its degree $n$ piece. Suppose $A$ has finite length, and let $\ell_i$ be the length of the graded piece $A_i$. The \defn{magnitude} of $A$ is the tuple $(\ell_0, \ell_1, \ell_2, \ldots)$, which is an element of $\bN^{\oplus \infty}$. We compare magnitudes using the reverse lexicographic order; this is a well-ordering, and permits us to use induction on magnitude. Such induction arguments are common in the theory of $\GL$-varieties, and will appear in \S \ref{s:core}.

Here is an important observation. Suppose $A$ is as above and concentrated in degrees $\le d$. Then the degree $d$ piece of the magnitude of $A$ and its restriction to $G(n)$ agree. Hence, if we restrict to $G(n)$ and then quotient by any subrepresentation in degree $d$, the magnitude drops. For instance, suppose $A=\Sym^2(\bV)$. Its magnitude is $(0, 0, 1, 0, \ldots)$. Restricting to $G(n)$ gives
\begin{displaymath}
\Sym^2(\bV) = \Sym^2(\bV_{>n}) \oplus (k^n \otimes \bV_{>n}) \oplus \Sym^2(k^n)
\end{displaymath}
where $\bV_{>n}$ is the subspace of $\bV$ spanned by the $e_i$ with $i>n$. As a representation of $G(n)$, this has magnitude $(\binom{n+1}{2}, n, 1, 0, \ldots)$. Thus if we take the above representation and quotient by $\Sym^2(\bV_{>n})$ the magnitude has nothing in degree~2, and is thus smaller than the initial magnitude.

\subsection{The shift theorem} \label{ss:shift}

We first recall some definitions and results from \cite{polygeom}. The following result is the ``embedding theorem'' \cite[\S 1.2]{polygeom}:

\begin{theorem} \label{thm:embedding}
Let $Y$ be a $\GL$-variety, let $R$ be an irreducible polynomial representation of $\GL$ of degree $d>0$, and let $X$ be a closed $\GL$-subvariety of $Y \times \bA(R)$. Then one of the following holds:
\begin{enumerate}
\item $X$ is cylindrical, that is, $X = Y_0 \times \bA(R)$ for some closed $\GL$-subvariety $Y_0 \subset Y$.
\item There is a non-zero $G(n)$-invariant function $h$ on $X$, for some $n$, such that the $G(n)$-variety $X[1/h]$ embeds into $Y \times \bA(W)$, where $W$ is a polynomial $G(n)$-representation such that each irreducible summand has degree $<d$.
\end{enumerate}
\end{theorem}

In fact, we can be more precise about the function $h$ appearing above, which will be important in our application. We record the statement here (see \cite[Theorem~4.1]{polygeom}):

\begin{proposition} \label{prop:embedding2}
Let $X$, $Y$, and $R$ be as in Theorem~\ref{thm:embedding}. Let $f$ be any function on $Y \times \bA(R)$ that vanishes on $X$, let $v \in R$, and let $h=\partial_v f$ be the corresponding partial derivative of $f$. Let $n$ be such that $f$ and $v$ (and thus $h$) are $G(n)$-invariant. Then the $G(n)$-variety $X[1/h]$ embeds into $Y \times \bA(W)$ for some $W$ as in case (b).
\end{proposition}

Note that in the above proposition, it is possible that $h$ vanishes identically on $X$, in which case $X[1/h]$ is empty. For example, this happens if $f$ factors through $Y$. Thus to profitably apply the proposition, one must take care in the selection of $f$ and $v$. To prove Theorem~\ref{thm:embedding}, we show that in the non-cylindrical case it is possible to make a choice where $h$ does not vanish identically on $X$; see the proof of \cite[Theorem~4.2]{polygeom}.

We say that a $\GL$-variety is \defn{elementary}\footnote{In \cite{polygeom}, we required $B$ to be irreducible, but we drop that condition here.} if it has the form $B \times \bA(V)$, where $B$ is a finite dimensional variety (with trivial $\GL$ action) and $V$ is a finite length polynomial representation of $\GL$; we similarly define the notion for $G(n)$-varieties. The embedding theorem essentially says that a $\GL$-variety becomes simpler after shifting and passing to an open subvariety. By iterating this procedure, we arrive at the following theorem, which is the ``shift theorem'' \cite[\S 1.2]{polygeom}.

\begin{theorem} \label{thm:shift}
Let $X$ be a non-empty $\GL$-variety. Then there is a non-zero $G(n)$-invariant function $h$ on $X$, for some $n$, such that $X[1/h]$ is an elementary $G(n)$-variety.
\end{theorem}

\begin{proof}
We sketch the proof, since the same idea will be used below. We prove the theorem for closed $\GL$-subvarieties of $\bA(V)$ by induction on the magnitude of $V$. Thus let $X \subset \bA(V)$ be given. If $V$ is concentrated in degree~0 then $X$ is finite dimensional and there is nothing to prove. Otherwise, write $V=V' \oplus R$ where $R$ is a top degree irreducible summand of $V$. Apply Theorem~\ref{thm:embedding} with $Y=\bA(V')$. There are two cases to consider.

In case (a), we have $X = Y_0 \times \bA(R)$, where $Y_0$ is a closed $\GL$-subvariety of $\bA(V')$. Since $V'$ has smaller magnitude than $V$, the theorem is true for $Y_0$ by the inductive hypothesis. Thus $Y_0[1/h]$ is an elementary $G(n)$-variety for some $n$ and $h$. Regarding $h$ as a function on $X$, we have $X[1/h] = Y_0[1/h] \times \bA(R)$, and so it too is elementary, as required.

In case (b), we see that $X[1/h]$ is a closed $G(n)$-subvariety of $\bA(V' \oplus W)$, where $W$ is a polynomial $G(n)$-representation concentrated in degrees strictly less than that of $R$. Thus $V' \oplus W$ has smaller magnitude than $V$, and so the result is true for $X[1/h]$ by the inductive hypothesis. We thus see that $X[1/hh']$ is $G(m)$-elementary, for some $m \ge n$ and some non-zero $G(m)$-invariant function $h'$ on $X[1/h]$, as required.
\end{proof}

\subsection{The core}

We now introduce the notion of core. Let $X$ be a $\GL$-variety.

\begin{definition}
A point $x \in X$ is \defn{elementary} if there exists an open elementary $G(n)$-subvariety of $X$ containing $x$, for some $n$. We let $X_{\elem}$ be the set of all elementary points. We define the \defn{core} of $X$, denoted $X_{\core}$ to be the complement of $X_{\elem}$.
\end{definition}

\begin{proposition} \label{prop:core}
The locus $X_{\elem}$ is $\GL$-stable, open, and dense. Hence $X_{\core}$ is a closed $\GL$-subvariety of $X$ that is nowhere dense.
\end{proposition}

\begin{proof}
If $x \in X_{\elem}$ then, by definition, there is an open elementary $G(n)$-subvariety $V$ of $X$ containing $x$. It follows that $V \subset X_{\elem}$, which shows that $X_{\elem}$ is open. Now suppose $g \in \GL$. Then for $m$ sufficiently large, $g$ centralizes $G(m)$ and $G(m)$ is contained in $G(n)$, and so $gV$ is a $G(m)$-stable open elementary subvariety. Since $gx \in gV$, we see that $gx \in X_{\elem}$, and so $X_{\elem}$ is $\GL$-stable. Finally, $X_{\elem}$ is dense by the shift theorem (Theorem~\ref{thm:shift}), applied separately to each irreducible component of $X$.
\end{proof}

\begin{example} \label{ex:fdcore}
Suppose $X$ is a finite dimensional variety with trivial $\GL$ action. Then $X$ is already elementary, and so the core of $X$ is empty.
\end{example}

One can iterate the core construction, i.e., take the core of $X$, then the core of that subvariety, and so on. This gives a strictly descending chain of closed $\GL$-subvarieties, which terminates by noetherianity. This essentially gives a canonical version of the locally elementary decomposition appearing in \cite[Corollary~7.9]{polygeom}; the one minor difference is that the elementary varieties used in \cite{polygeom} are required to be irreducible.

\subsection{The main result}

We now prove our main result on the core, following two lemmas. The following is simply \cite[Lemma~3.7]{CDD}, reformulated in the language of $\GL$-varieties. It is the key result.

\begin{lemma} \label{lem:core-sing-1}
Let $V$ be a finite length polynomial representation of $\GL$ concentrated in degrees $\le d$ and for which $V_d \ne 0$. Let $X$ be an irreducible closed $\GL$-subvariety of $\bA(V)$. Make the following assumption:
\begin{itemize}
\item[$(\ast)$] For every choice of decomposition $V=V' \oplus R$, with $R$ irreducible of degree $d$, the variety $X$ is non-cylindrical, i.e., not of the form $Y \times \bA(R)$ with $Y$ a closed subvariety of $\bA(V')$.
\end{itemize}
Let $Z$ be the closed $\GL$-subvariety of $X$ defined by the equations $\partial_v f =0$, where $f$ varies over all functions that vanish on $X$ and $v$ varies over all elements of $V_d$ that generate an irreducible subrepresentation. Then $Z$ is contained in $X_{\sing}$.
\end{lemma}

In fact, the proof of the lemma shows that it is enough to know that $X$ is non-cylindrical for a single choice of decomposition, but this will not matter to us. We next show that one can upgrade the embedding theorem (Theorem~\ref{thm:embedding}) to obtain some control on the open set it produces. This same idea was used in \cite[\S 3.5]{CDD}.

\begin{lemma} \label{lem:core-sing-2}
Let $V$ be a finite length polynomial representation of $\GL$ concentrated in degrees $\le d$ and for which $V_d \ne 0$. Let $X$ be an irreducible closed $\GL$-subvariety of $\bA(V)$, and let $x \in X_{\ns}$. Then one of the following holds:
\begin{enumerate}
\item There exists a decomposition $V=V' \oplus R$, where $R$ is irreducible of degree $d$, for which $X$ is cylindrical, i.e., $X=Y \times \bA(R)$ for some closed $\GL$-subvariety $Y \subset \bA(V')$.
\item There exists a non-zero $G(n)$-invariant function $h$ on $X$ such that $X[1/h]$ contains $x$ and embeds into $\bA(W)$ for some finite length polynomial $G(n)$-representation with smaller magnitude than $V$.
\end{enumerate}
\end{lemma}

\begin{proof}
Suppose we are not in case~(a). This means condition $(\ast)$ from Lemma~\ref{lem:core-sing-1} holds. Let $Z$ be as in that lemma. Since $x$ is non-singular, we have $x \not\in Z$. This means there is some function $f$ on $\bA(V)$ that vanishes on $X$ and some $v \in V_d$ that generates an irreducible subrepresentation $R$ of $V$ such that $h=\partial_v f$ does not vanish at $x$. Let $V'$ be a complementary subrepresentation to $R$. We now apply Proposition~\ref{prop:embedding2}, with $Y=\bA(V')$. This produces the required embedding of $X[1/h]$.
\end{proof}

We are ready to prove Theorem~\ref{mainthm4}.

\begin{theorem}
\label{thm:core-embcodim}
Let $X$ be a $\GL$-variety. Then the core of $X$ is given as
\[
	X_{\core} = \{x \in X \mid \embcodim(\cO_{X,x}) = \infty\}.
\]
\end{theorem}

\begin{remark}
It is possible for a singular point of a $\GL$-variety to have any finite embedding codimension. Indeed, this can happen for ordinary finite dimensional varieties, and such varieties can be regarded as $\GL$-varieties with trivial $\GL$-action. The theorem asserts that any such point is elementary.
\end{remark}

\begin{remark}
\label{rem:gl-vs-arcs}
The reader should compare this statement with one of the main results in \cite{CdFD} in the setting of arc spaces. Let $X_\infty$ be the arc space of a variety $X$ over a perfect field $k$. Then $\alpha \in X_\infty$ satisfies $\embcodim(\cO_{X_\infty,\alpha}) < \infty$ if and only if $\alpha \notin (\Sing X)_\infty$. Following a result by Drinfeld, in \cite{Chi}, the latter condition was shown to imply to the existence of a morphism of formal schemes (up to finite separable extension of coefficient fields)
\[
	\Spf( \widehat{\cO_{X_\infty,\alpha}}) \cong \widehat{Y}_y \times \widehat{\bA^\infty}_z.
\]
Here $Y$ is a scheme of finite type and $\widehat{\bA^\infty}_z$ denotes the formal neighborhood of an infinite-dimensional affine space at some point $z$. Conversely, the existence of such a decomposition necessarily implies that $\embcodim(\cO_{X_\infty,\alpha})<\infty$. Summarizing, there is somewhat of a similarity between the statement of Theorem~\ref{thm:core-embcodim} and the above result for arc spaces: in both cases, finiteness of the embedding codimension determines the existence locally of a decomposition into a finite-dimensional (potentially singular) piece and an infinite-dimensional smooth piece. However, the situation of $\GL$-varieties is much more rigid, in that the decomposition is forced to exist Zariski-locally instead of just on the level of formal neighborhoods.
\end{remark}


\begin{proof}
If $x$ is elementary then $\embcodim(\cO_{X,x}) < \infty$ by \cite[Proposition 7.6]{CdFD}. Thus we are left to show that for every $x \in X_{\core}$ we have $\embcodim (\cO_{X,x}) = \infty$. We assume that $X$ is a closed  $\GL$-subvariety of some $\bA(V)$ for a polynomial representation $V$ and prove the statement by induction on the magnitude $m$ of $V$. If $m$ is concentrated in degree $0$ then the result trivially holds: the core of $X$ is empty and every point of $X$ has finite embedding codimension.

Now assume first that there exists a decomposition $V = V' \oplus R$ with $R$ irreducible of degree $d$ and such that $X$ is cylindrical with respect to this decomposition, that is $X = Y \times \bA(R)$ for some closed $\GL$-subvariety $Y \subset \bA(V')$. If $y$ denotes the image of $x$ in $Y$, then again by \cite[Proposition 7.6]{CdFD} we have
\[
	\embcodim(\cO_{X,x}) = \embcodim(\cO_{Y,y})
\]
and $x$ is elementary if and only if $y$ is. As the inclusion holds for $Y$ by induction, it therefore holds for $X$ as well.

Next, assume that $x$ is not in the $\GL$-subvariety $Z$ of Lemma~\ref{lem:core-sing-1}. That is, there exists a function $f$ vanishing on $X$ and an element $v$ of $V_d$ with the property that, if $h = \frac{\partial f}{\partial v}$, then $x \in X[1/h]$ and $X[1/h]$ embeds into $\bA(W)$ for some finite length polynomial $G(n)$-representation of smaller magnitude than $V$. As before, we have that the inclusion holds for $X[1/h]$ by induction, and clearly $x$ is elementary if and only if its image in $X[1/h]$ is. The embedding codimension being Zariski-local we have
\[
	\embcodim(\cO_{X,x}) = \embcodim(\cO_{X[1/h],x}),
\]
thus the inclusion holds in this case as well.

We are left with the case $x \in Z$. For convenience, given any $\GL$-variety $Y$ and $y \in Y$ we will write $Y_n = Y\{k^n\}$ and $y_n$ for the image of $y$ in $Y_n$.  Consider the decomposition $V = V_{<d} \oplus V_d$ and the induced projection map $X \to \bA(V_d)$. We first claim that the difference
\[
	\dim \bA(V_d)_n - \dim X_n
\]
grows polynomially of degree $d$ in $n$, so in particular becomes arbitrarily large for $n \gg 0$. Indeed, this is the main content of \cite[Lemma 3.7]{CDD}, but we sketch the argument for the convenience of the reader. By Lemma~\ref{lem:core-sing-2} there exists a dense open subset of $X$ which embeds into $\bA(W)$ for some finite-length polynomial representation $W$ of smaller magnitude than $V$. This implies that the difference 
\[
	\dim \bA(V_d)_n - \dim \bA(W)_n
\]
is polynomial of degree $d$ in $n$. Since $\dim \bA(W)_n$ bounds the dimension of $X_n$ the result follows.

Write $x'$ for the image of $x$ in $\bA(V_d)$ and similar $x_n$ and $x'_n$ for the images in $X_n$ and $\bA(V_d)_n$. Consider the diagram
\[
	\xymatrix@C=0.3em{ & 0 \ar[rr] && \fm_{x_n}/\fm_{x_n}^2 \otimes k(x) \ar[ld]^{\pi_n} \ar[rr] && \Omega_{X_n,x_n} \otimes k(x) \ar[ld]  \\
	0 \ar[rr] && \fm_x/\fm_x^2 \ar[rr] && \Omega_{X,x} \otimes k(x)  & \\
	& 0 \ar[rr]|!{[ul];[ur]}\hole && \fm_{x'_n}/\fm_{x'_n}^2 \otimes k(x) \ar[uu]_<<<<<<{\varphi_n}|\hole \ar[ld]^{\pi'_n} \ar[rr]|\hole && \Omega_{\bA(V_d)_n,x'_n} \otimes k(x) \ar[uu]_<<<<<<{\psi_n} \ar[ld] \\
			0 \ar[rr] && \fm_{x'}/\fm_{x'}^2 \otimes k(x) \ar[rr] \ar[uu]_<<<<<<{\varphi} && \Omega_{\bA(V_d),x'} \otimes k(x) \ar[uu]_<<<<<<{\psi}.  &}
\]
The condition that $x \in Z$ implies that all partial derivatives of the form $\frac{\partial f}{\partial v}$ for $f$ in the defining ideal of $X$ and $v\in V_d$ generating an irreducible subrepresentation of $V$ vanish at $x$. Since the base field $k$ is of characteristic $0$ we have that $V_d$ is the sum of its irreducible subrepresentations and hence $\bA(V_d)$ is the spectrum of a polynomial ring with variables given by all such $v$ as before. Choosing presentations for $\Omega_{\bA(V_d)}$ and $\Omega_X$ we conclude that $\psi$ and $\psi_n$ are injective. From the above diagram it follows that $\varphi$ and $\varphi_n$ are injective. Since $\pi'_n$ is clearly injective we have
\[
	\dim  \fm_{x'_n}/\fm_{x'_n}^2 \otimes k(x) = \dim \cO_{\bA(V_d)_n,x'_n}.
\] 
Therefore the linear map $\pi_n$ has rank at least $\dim \cO_{\bA(V_d)_n,x'_n}$. Applying \cite[Proposition 6.6]{CdFD} yields
\begin{multline*}
	\embcodim(\cO_{X,x}) \geq \dim \cO_{\bA(V_d)_n,x'_n} - \dim \cO_{X_n,x_n} = \\ = \dim (\bA(V_d)_n) - \dim(X_n) + \underbrace{\dim(\overline{\{x_n\}}) - \dim(\overline{\{x'_n\}})}_{\geq 0},
\end{multline*}
and so $\embcodim(\cO_{X,x}) = \infty$ by the discussion above.
\end{proof}

From the characterization provided by Theorem~\ref{thm:core-embcodim} we immediately obtain that the core consists only of singular points.

\begin{corollary}
\label{thm:core-sing}
Let $X$ be a $\GL$-variety. Then $X_{\core} \subseteq X_\sing$.
\end{corollary}

\begin{proof}
Let $x\in X_{\ns}$, which means that there exists $n_0$ such that $X\{k^n\}$ is regular at $\pi_n(x)$ for all $n \geq n_0$, where $\pi_n \colon X \to X\{k^n\}$. Now the local ring $\cO_{X,x}$ is the directed colimit of the local rings $\cO_{X\{k^n\},\pi_n(x)}$ and hence the claim follows from Lemma~\ref{lem:colimit-regular} together with Theorem~\ref{thm:core-embcodim}.
\end{proof}

Corollary~\ref{thm:core-sing} in turn implies the following important result about the local structure of a $\GL$-variety at a non-singular point. Note that this establishes (a) $\Rightarrow$ (f) in Theorem~\ref{mainthm}.

\begin{corollary} \label{cor:core-sing}
A non-singular point $x$ of a $\GL$-variety $X$ admits a $G(n)$-stable open elementary neighborhood that is itself non-singular, for some $n$.
\end{corollary}

\begin{proof}
By Theorem~\ref{thm:core-sing}, $x$ belongs to the elementary locus $X_{\elem}$, and thus admits a $G(n)$-stable open neighborhood $U$ that is elementary as a $G(n)$-variety. Write $U=B \times \bA(V)$, where $B$ is a finite dimensional variety with trivial $G(n)$-action and $V$ is a finite length polynomial representation of $G(n)$. Since $x$ is a non-singular point of $U$, its $B$ component must be a non-singular point of $B$. Thus, replacing $B$ with an appropriate affine open, we can assume that $B$ is non-singular, and then $U$ is as well.
\end{proof}

%
%

\section{The main theorem}

In this section, we prove most of Theorem~\ref{mainthm}.

\subsection{Finite generation of ideals}

The preeminent question in the theory of $\GL$-algebras is noetherianity: is every $\GL$-ideal of a finitely $\GL$-generated $\GL$-algebra necessarily finitely $\GL$-generated? We cannot prove this, but we observe that we can obtain a useful result in this direction as a consequence of the shift theorem. We first require a lemma. We say that a $\GL$-algebras is \defn{finitely $\GL$-presented} if it has the form $S/\fa$ where $S$ is the symmetric algebra on a finite length polynomial representation and $\fa$ is a finitely $\GL$-generated ideal of $S$.

\begin{lemma} \label{lem:fp}
Let $f \colon A \to B$ be a surjective homomorphism of $\GL$-algebras, with $A$ finitely $\GL$-generated and $B$ finitely $\GL$-presented. Then $\ker(f)$ is finitely $\GL$-generated.
\end{lemma}

\begin{proof}
Choose a surjection $a \colon S \to A$ where $S$ is the symmetric algebra on a finite length polynomial representation $V$, which exists since $A$ is finitely $\GL$-generated. Also choose a surjection $b \colon R \to B$, where $R$ is the symmetric algebra on a finite length polynomial representation $W$ and $\ker(b)$ is finitely $\GL$-generated, which exists since $B$ is finitely $\GL$-presented. Since polynomial representations are semi-simple, the map $b \colon W \to B$ lifts through the surjection $fa$. This gives us a map of $\GL$-algebras $i \colon R \to S$ such that $fai = b$. Similarly, we have a map of $\GL$-algebras $j \colon S \to R$ such that $bj= fa$. Consider the diagram
\begin{displaymath}
\xymatrix@C=4em{
S \otimes_k R \ar[r]^-{\id \otimes i} \ar[d]_{j \otimes \id} & S \ar[r]^a & A \ar[d]^f \\
R \ar[rr]^b && B }
\end{displaymath}
This diagram commutes, and all maps are surjections of $\GL$-algebras. It suffices to show that the kernel $\fc$ of the composition $fa(\id \otimes i)$ is finitely $\GL$-generated, since this surjects onto the kernel of $f$. Since the diagram commutes, $\fc$ is equal to the kernel of $b(j \otimes id)$. This is finitely $\GL$-generated since both $\ker(b)$ and $\ker(j \otimes \id)$ are; indeed, $\ker(j \otimes \id)$ is generated by the image of the map $V \to S \otimes_k R$ given by $x \mapsto x \otimes 1 - 1 \otimes j(x)$.
\end{proof}

The following is our result on finite generation of ideals.

\begin{proposition} \label{prop:fg}
Let $A$ be a finitely $\GL$-generated $\GL$-algebra and let $\fp$ be a $\GL$-stable prime ideal of $A$. Then there exists an integer $n \ge 0$ and a $G(n)$-invariant function $h \in A \setminus \fp$ such that the extension of $\fp$ to $A[1/h]$ is finitely $G(n)$-generated.
\end{proposition}

\begin{proof}
Let $B=A/\fp$. By the shift theorem (Theorem~\ref{thm:shift}), there is a non-zero $G(n)$-invariant element $\ol{h} \in B$, for some $n$, such that $\Spec(B[1/\ol{h}])$ is an elementary $G(n)$-variety. This means that $B[1/\ol{h}] = B_1 \otimes_k B_2$, where $B_1$ is a finitely generated $k$-algebra (with trivial $G(n)$ action) and $B_2$ is the symmetric algebra on a finite length polynomial representation. It follows that $B[1/\ol{h}]$ is finitely $G(n)$-presented. Let $h \in A$ be any $G(n)$-invariant lift of $\ol{h}$. Then the map $A[1/h] \to B[1/\ol{h}]$ is a surjection of $G(n)$-algebras with kernel $\fp[1/h]$. The result now follows from Lemma~\ref{lem:fp}.
\end{proof}

\subsection{Regular points are non-singular} \label{ss:reg-ns}

Fix a $\GL$-variety $X$. Our goal at the moment is to establish (b) $\Rightarrow$ (a) in Theorem~\ref{mainthm}, i.e., that if $x \in X$ is a point with $\cO_{X,x}$ regular then $x \in X_{\ns}$. We set-up some notation. Let $A$ be the coordinate ring of $X$. Write $X_n$ and $A_n$ for the evaluations on $k^n$, and similarly for ideals of $A$. Recall that a map of schemes is submersive at a point if the map on Zariski cotangent spaces is injective (Definition~\ref{defn:subm}).

\begin{lemma} \label{lem:mainthm-1}
Suppose $x \in X$ is $\GL$-invariant. Then there exists $n_0$ such that the map $X \to X\{k^n\}$ is submersive at $x$, for all $n \ge n_0$.
\end{lemma}

\begin{proof}
Let $\fp$ be the prime ideal of $A$ corresponding to $x$. We are free to shift, since that only affects the resulting value of $n_0$. We are also free to invert elements of $A \setminus \fp$, since this does not change the local ring at $\fp$. We may thus assume that $\fp$ is finitely $\GL$-generated (Proposition~\ref{prop:fg}).

Let $B=A/\fp$, so that $\fp/\fp^2$ is a finitely $\GL$-generated $B$-module. By generic freeness (Theorem~\ref{thm:genfree}), we can shift and localize such that $\fp/\fp^2$ is equivariantly free as a $B$-module, i.e., isomorphic to a module of the form $B \otimes_k V$ for some finite length polynomial representation $V$. For any $n \ge 0$, we have a commutative diagram
\begin{displaymath}
\xymatrix{
\kappa(\fp) \otimes_{A_n} \fp_n/\fp_n^2 \ar[r] \ar@{=}[d] &
\kappa(\fp) \otimes_A \fp/\fp^2 \ar@{=}[d] \\
\kappa(\fp) \otimes_k V_n \ar[r] & \kappa(\fp) \otimes_k V }
\end{displaymath}
and so the top map is injective, as required.
\end{proof}

Suppose $X$ is a closed $\GL$-subvariety of $\bA(V)$. Then the jet space $J_m(X)$ is, by construction, a closed subscheme of $\bA(V^{\oplus (m+1)})$ that is $\GL$-stable. It is thus the spectrum of a $\GL$-algebra that is finitely $\GL$-generated over $k$. It follows that the reduced subscheme of $J_m(X)$ is a $\GL$-variety. This is a very useful observation, since it allows us to apply the theory of $\GL$-varieties to the study of jets on $\GL$-varieties. The following lemma is proved using this strategy.

\begin{lemma} \label{lem:mainthm-2}
Let $x,y \in X$ belong to the same generalized $\GL$-orbit. Then:
\begin{enumerate}
\item $\cO_{X,x}$ has the jet lifting property if and only if $\cO_{Y,y}$ does.
\item $\cO_{X,x}$ is regular if and only if $\cO_{X,y}$ is regular.
\end{enumerate}
\end{lemma}

\begin{proof}
Let $W_m \subset X$ be the set of points such that every $m$-jet at $x$ lifts to an $(m+1)$-jet. Consider the maps
\begin{displaymath}
\xymatrix@C=3em{
J_{m+1}(X) \ar[r]^-{f_m} & J_m(X) \ar[r]^-{g_m} & X }
\end{displaymath}
Then $J_m(X) \setminus \im(f_m)$ is the locus of $m$-jets that fail to lift to an $(m+1)$-jet, and so $W_m$ is complement of the image of this set under $g_m$. This is $\GL$-constructible by the $\GL$-version of Chevalley's theorem \cite[\S 1.2]{polygeom}; we recall that $\GL$-constructible means a finite union of locally closed $\GL$-stable subsets. Note that the potentially non-reducedness of jet schemes is not an issue here, since we can simply pass to the reduced subschemes. Let $W=\bigcap_{m \ge 1} W_m$; this is the locus of points that have the jet lifting property. Since each $W_m$ is a union of generalized $\GL$-orbits, so is $W$, and so (a) follows. Regular is equivalent to the jet lifting property (Proposition~\ref{prop:compare}), and so (b) follows from (a).
\end{proof}

We now come to the main point of \S \ref{ss:reg-ns}.

\begin{proposition} \label{prop:mainthm-3}
Let $x \in X$. If $\cO_{X,x}$ is regular then $x \in X_{\ns}$.
\end{proposition}

\begin{proof}
First suppose that $x$ is $\GL$-invariant. By Lemma~\ref{lem:mainthm-1}, there is some $n_0$ such that the map $\pi_n \colon X \to X\{k^n\}$ is submersive at $x$ for all $n \ge n_0$. By Lemma~\ref{prop:reg-contract} this implies that $X\{k^n\}$ is regular at $\pi_n(x)$ for all $n\geq n_0$, hence $x\in X_{\ns}$.


We now treat the general case. Let $y$ be the unique $\GL$-invariant point in the generalized orbit of $x$. Since $\cO_{X,x}$ is regular, so is $\cO_{X,y}$ (Lemma~\ref{lem:mainthm-2}). By the previous paragraph, $y \in X_{\ns}$. Since $X_{\ns}$ is a $\GL$-stable open set, it contains the entire generalized $\GL$-orbit of $y$, and so $x$ belongs to $X_{\ns}$ as required.
\end{proof}

\subsection{The theorem}

At this point, we have effectively proved Theorem~\ref{mainthm}, except for the part involving differentials (which is handled in \S \ref{s:omega} below). We write out the details to explain how the logic fits together, and restate the theorem for convenience.

\begin{theorem}
The following conditions are equivalent, for a point $x$ of a $\GL$-variety $X$:
\begin{enumerate}
\item The point $x$ is non-singular, i.e., it belongs to $X_{\ns}$.
\item The local ring $\cO_{X,x}$ is regular (Definition~\ref{defn:regular}).
\item The jet lifting property holds at $x$ (Definition~\ref{defn:jlp}).
\item The local ring $\cO_{X,x}$ is a formally smooth $k$-algebra (Definition~\ref{defn:fsm}).
\item The local ring $\cO_{X,x}$ is a localization of a ring of the form $A[t_i]_{i \in I}$, where $A$ is a smooth $k$-algebra and $I$ is an index set.
\item There is non-singular $G(n)$-stable elementary open neighborhood of $x$, for some $n$.
\end{enumerate}
\end{theorem}

\begin{proof}
We have seen that (a) implies (f) (Corollary~\ref{cor:core-sing}), and it is clear that (f) implies (e). It is well-known that that (e) implies (d): indeed, smooth algebras are formally smooth, polynomial rings over formally smooth algebras are formally smooth, and a localization of a formally smooth algebra is formally smooth. We have seen (Proposition~\ref{prop:compare}) that (b) and (c) are equivalent, and they are implied by (d). Finally, (b) implies (a) by Proposition~\ref{prop:mainthm-3}, which completes the proof.
\end{proof}

\subsection{A quantitative result}

The methods used above allow us to prove a quantitative result on singularities of $\GL$-varieties, which we now explain. Let $X$ be a $k$-scheme and let $x \in X$ be a point for which the jet lifting property fails. Define $\sigma(x)$ to be the minimum $m$ for which an $m$-jet at $x$ fails to lift to an $(m+1)$-jet. When $X$ is a hypersurface embedded in a finite-dimensional nonsingular variety, then $\sigma(x)$ is just the multiplicity of $X$ at $x$. Hence $\sigma(x)$ can be regarded as a measure of singularities.


\begin{theorem} \label{thm:quant}
Let $X$ be a $\GL$-variety. Then there exist an integer $\sigma_0$ such that $\sigma(x) \le \sigma_0$ for all $x \in X_{\sing}$.
\end{theorem}

\begin{proof}
Let $W_m \subset X$ be the locus of points where some tangent vector at $x$ fails to lift to an $m$-jet. This set is $\GL$-constructible, by the same argument appearing in the proof of Lemma~\ref{lem:mainthm-2}. We have $X_{\sing} = \bigcup_{m \ge 1} W_m$ by Theorem~\ref{mainthm} and Remark~\ref{rmk:1-jet}. Since $W_{\bullet}$ is an ascending chain of $\GL$-constructible sets and $X_{\sing}$ is closed, it follows that $X_{\sing}=W_m$ for some $m$. We can then take $\sigma_0=m$.
\end{proof}

Thus, in a rough sense, the singular points of $X$ are all (approximately) equally singular. 
We expect a similar statement to hold for the system of finite dimensional varieties $X\{k^n\}$, though we have not worked out the details.

%
%

%
%

\section{The flat locus of a module} \label{s:omega}

In this section, we prove Theorem~\ref{mainthm3}, and use it to finish up the final detail in the proof of Theorem~\ref{mainthm}.

\subsection{The main result and an application}

The following is a slightly more precise version of Theorem~\ref{mainthm3}.

\begin{theorem} \label{thm:flat}
Let $A$ be a reduced finitely $\GL$-generated $\GL$-algebra, let $X=\Spec(A)$, and let $M$ be a finitely $\GL$-generated $A$-module. Let $U \subset X$ be the set of prime ideals $\fp$ such that $M_{\fp}$ is a flat $A_{\fp}$-module. Then:
\begin{enumerate}
\item The set $U$ is open and dense in $X$.
\item For each $\fp \in U$, the $A_{\fp}$-module $M_{\fp}$ is free.
\item Let $Z=X \setminus U$ be the non-flat locus and let $W_n \subset X\{k^n\}$ be the non-flat locus of $M\{k^n\}$. Then $W_n$ is contained in $Z\{k^n\}$, with equality for $n \gg 0$.
\end{enumerate}
\end{theorem}

The proof of the theorem will take the remainder of \S \ref{s:omega}. Before getting to it, we look at an application where we take $M$ to be the module of differentials.

We first make some general comments about differentials. Let $A$ be a $\GL$-algebra and let $I$ be the kernel of the multiplication map $A \otimes A \to A$. Then the module $\Omega_A$ of K\"ahler differentials is $I/I^2$, and is thus clearly a module over $A$ in our sense of the term, i.e., a module internal to the category of polynomial representations. Moreover, since evaluation is exact and compatible with tensor products, we see that $I\{k^n\}$ is the kernel of $A\{k^n\} \otimes A\{k^n\} \to A\{k^n\}$, and so $\Omega_A\{k^n\}$ is identified with $\Omega_{A\{k^n\}}$.

We now come to our main application of the theorem. Note that this establishes (a) $\Leftrightarrow$ (g) in Theorem~\ref{mainthm}, and completes the proof of that theorem.

\begin{corollary} \label{cor:sing-omega}
Let $X$ be a $\GL$-variety. Then $X_{\sing}$ is exactly the non-flat locus of $\Omega_X$.
\end{corollary}

\begin{proof}
Let $Z$ be the non-flat locus of $\Omega_X$. For $n$ sufficiently large, $Z\{k^n\}$ is the non-flat locus of $\Omega_X\{k^n\}=\Omega_{X\{k^n\}}$, which is the singular locus of $X\{k^n\}$, which is $X_{\sing}\{k^n\}$. We thus see that $Z\{k^n\}=X_{\sing}\{k^n\}$ for all $n \gg 0$, and so $Z=X_{\sing}$, as required.
\end{proof}

\begin{remark} \label{rmk:newproof}
The above argument gives a new proof of the main theorem of \cite{CDD}. Indeed, it shows that there is a closed $\GL$-subvariety of $X$ (namely $Z$) that evalutes to the singular locus of $X\{k^n\}$ for all $n \gg 0$. However, as we see below, this proof relies on the version of the Krull intersection theorem proved in \cite{imgclosure}, which is a fairly difficult theorem.
\end{remark}

\subsection{Nakayama's lemma}

We now prove some versions of Nakayama's lemma for $\GL$-algebras. The following is the most straightforward version.

\begin{proposition} \label{prop:nak1}
Let $A$ be a finitely $\GL$-generated $\GL$-algebra, let $\fp$ be a $\GL$-stable prime ideal of $A$, and let $M$ be a finitely $\GL$-generated $A$-module such that $M_{\fp}/\fp M_{\fp}=0$. Then $M_{\fp}=0$.
\end{proposition}

\begin{proof}
Let $x_1, \ldots, x_n$ be $\GL$-generators of $M$. For each $1 \le i \le n$, there is an expression $s_i x_i = \sum_{i=1}^n \sum_j a_{i,j} g_j x_i$, where $s_i \not\in \fp$ and $a_{i,j} \in \fp$ and $g_j \in \GL$. Fix $N>0$, and use a prime to denote evaluation on $k^N$, e.g., $A'=A\{k^N\}$. Take $N$ sufficiently large so that each $s_i$ and $a_{i,j}$ belongs to $A'$, and each $g_j$ belongs to $\GL_N$. Since $x_1, \ldots, x_n$ are $\GL_N$-generators for $M'$, we thus see that $M'_{\fp'}/\fp' M'_{\fp'}=0$. By the classical Nakayama lemma, it follows that $M'_{\fp'}=0$. Thus for each $1 \le i \le n$ there is some $s'_i \in A' \setminus \fp'$ such that $s'_i x_i=0$. This shows that each $x_i$ vanishes in $M_{\fp}$, and so $M_{\fp}=0$.
\end{proof}

In our main application, we will need to apply Nakayama's lemma to a module $M$ that occurs as a submodule of $F$, where $F$ has the form $V \otimes A$ and $V$ is a finite length polynomial representation. A general expectation (or hope) is that a finitely $\GL$-generated $A$-module is noetherian. If this were the case then $M$ would be finitely $\GL$-generated, and we could use the above version of Nakayama's lemma. However, this noetherian statement is not known. Fortunately, we can still establish Nakayama's lemma in this setting. The key tool is the following version of the Krull intersection theorem from \cite[Theorem~1.6]{imgclosure}:

\begin{theorem} \label{thm:krull}
Let $A$ be a reduced finitely $\GL$-generated $\GL$-algebra and let $\fp$ be a $\GL$-stable prime ideal of $A$. Then $\bigcap_{n \ge 0} \fp^n A_{\fp}=0$.
\end{theorem}

And here is our alternate version of Nakayama's lemma. Note that we do not actually need $V$ to be finite length.

\begin{proposition} \label{prop:nak2}
Let $A$ be a reduced finitely $\GL$-generated $\GL$-algebra and let $\fp$ be a $\GL$-stable prime ideal of $A$. Let $V$ be a polynomial representation, let $F=A \otimes_k V$, and let $M$ be an $A$-submodule of $F$. Suppose $M_{\fp}/\fp M_{\fp}=0$. Then $M_{\fp}=0$.
\end{proposition}

\begin{proof}
First notice that $\bigcap_{n \ge 1} \fp^n F_{\fp}=0$. Indeed, if we pick a basis of $V$ then the coefficients of an element of this module belong to $\bigcap_{n \ge 1} \fp^n$, and thus vanish (Theorem~\ref{thm:krull}). It follows that $\bigcap_{n \ge 1} \fp^n M_{\fp}=0$ too, as $\fp^n M_{\fp} \subset \fp^n F_{\fp}$ for each $n$. Now, suppose $M_{\fp}/\fp M_{\fp}=0$. This means that $M_{\fp}=\fp M_{\fp}$, and so $M_{\fp}=\fp^n M_{\fp}$ for each $n \ge 1$. Thus $M_{\fp} = \bigcap_{n \ge 1} \fp^n M_{\fp}=0$.
\end{proof}

In fact, the conclusion of the above proposition can be improved:

\begin{proposition} \label{prop:nak3}
In the above setting, there is $f \in A \setminus \fp$ such that $M[1/f]=0$.
\end{proposition}

\begin{proof}
We can discard the irreducible components of $\Spec(A)$ not containing $\fp$; in other words, we can assume that $\fp$ contains every minimal prime. After making this reduction, we must show $M=0$. Let $\{e_i\}_{i \in I}$ be a basis of $F$ and let $x=\sum_{i \in I} a_i e_i$ be an element of $M$. Since $M_{\fp}=0$, there is some $s \in A \setminus \fp$ such that $sx=0$ holds in $M$, meaning $sa_i=0$ in $A$ for all $i$. But $s$ does not belong to any minimal prime of $A$, and so $s$ is a non-zerodivisor in $A$. Thus $a_i=0$ for all $i$, and so $x=0$.
\end{proof}

\subsection{Proof of the main theorem}

We prove Theorem~\ref{thm:flat} in a sequence of lemmas. Fix notation as in the theorem, i.e., $A$ is a reduced finitely $\GL$-generated $\GL$-algebra, $X=\Spec(A)$, $M$ is a finitely $\GL$-generated $A$-module, and $U \subset X$ is the set of primes $\fp$ such that $M_{\fp}$ is a flat $A_{\fp}$-module.

\begin{lemma} \label{lem:flat1}
Let $\fp$ be a $\GL$-stable prime in $U$. Then $U$ contains an open neighborhood $V$ of $\fp$, and $M$ is locally free over $V$. (We do not require this $V$ to be $\GL$-stable.)
\end{lemma}

\begin{proof}
Let $B=A/\fp$, which is an integral $\GL$-algebra. Consider the $B$-module $N=M/\fp M$. By generic freeness (Theorem~\ref{thm:genfree}), there is a non-zero $G(n)$-invariant element $h \in B$ such that $N[1/h]$ is free as a $G(n)$-equivariant $B[1/h]$-module, i.e., there is a finite length $G(n)$-subrepresentation $W \subset N[1/h]$ such that the natural map $B[1/h] \otimes W \to N[1/h]$ is an isomorphism. For notational simplicity, pick $\tilde{h} \in A$ lifting $h$, and replace all objects by their localizations at $\tilde{h}$. Thus we now have $W \subset N$ such that $B \otimes W \to N$ is an isomorphism.

Let $V$ be a $G(n)$-subrepresentation of $M$ mapping isomorphically to $W$. Let $F=A \otimes V$, let $F \to M$ be the natural map, and let $K$ be its kernel. Note that $F/\fp F \to M/\fp M$ is an isomorphism, and so by Nakayama's lemma (Proposition~\ref{prop:nak1}), $F_{\fp} \to M_{\fp}$ is surjective. We thus see that the sequence
\begin{displaymath}
0 \to K_{\fp} \to F_{\fp} \to M_{\fp} \to 0
\end{displaymath}
is exact. Since $M_{\fp}$ is flat, it follows that the sequence is exact after reducing modulo $\fp$. We thus find $K_{\fp}/\fp K_{\fp}=0$, and so $K_{\fp}=0$ by the variant of Nakayama's lemma (Proposition~ \ref{prop:nak2}). Moreover, by shifting and inverting an element outside of $\fp$, we have $K=0$ (Proposition~\ref {prop:nak3}). Thus $F \to M$ is injective, and $(M/F)_{\fp}=0$. Since $M$ is finitely generated, there is a single $h \not\in \fp$ such that $(M/F)[1/h]=0$. We thus see that $F \to M$ is an isomorphism over an open set $V$ containing $\fp$, and so $V \subset U$ as required.
\end{proof}

\begin{lemma} \label{lem:flat2}
The set $U$ is open, and $M$ is locally free over $U$.
\end{lemma}

\begin{proof}
Let $\fp \in U$. Let $Z \subset X$ be the $\GL$-orbit closure of $\fp$, which is an irreducible $\GL$-variety, and let $\fq$ be the generic point of $Z$, which is a $\GL$-stable prime ideal of $A$. Since $\fp \in Z$ we have $\fq \subset \fp$, and so $M_{\fq}$ is a localization of $M_{\fp}$. Since $M_{\fp}$ is flat, it follows that $M_{\fq}$ is flat. By Lemma~\ref{lem:flat1}, we see that there is an open neighborhood $V$ of $\fq$ that is contained in $U$, and $M$ is locally free over $V$. Of course, $U$ is $\GL$-stable, and so we can assume that $V$ is as well (just replace $V$ by the union of its $\GL$-translates). Since $V$ is a $\GL$-stable open set containing $\fq$, it contains the entire generalized orbit of $\fp$, and in particular contains $\fq$. This completes the proof.
\end{proof}

\begin{lemma}
The set $U$ is dense in $X$.
\end{lemma}

\begin{proof}
This follows from generic freeness (Theorem~\ref{thm:genfree}), applied separately on each irreducible component of $X$.
\end{proof}

We have thus proved parts (a) and (b) of Theorem~\ref{thm:flat}. We now turn to (c). Let $Z$ be the complement of $U$ in $X$, which is a closed $\GL$-subvariety. Wrtie $X_n$ for $X\{k^n\}$, and similarly for other objects. Let $W_n \subset X_n$ be the locus where $M_n$ is not flat. Our goal is to compare $Z$ and the $W_n$'s. We begin with a general observation.

\begin{lemma} \label{lem:flat-invar}
Let $B$ be a $\GL$-algebra and let $N$ be a $B$-module. If $N$ is flat over $B$ then $N^{\GL}$ is flat over $B^{\GL}$.
\end{lemma}

\begin{proof}
Let $Q$ be an arbitrary $B_0$-module, and regard it as a $B$-module via the identification $B_0=B/B_+$. Since $N$ is $B$-flat, we have $\Tor^B_i(N, Q)=0$ for all $i>0$. Taking $\GL$-invariants, which is compatible with $\Tor$, we find $\Tor^{B_0}_i(N_0, Q)=0$ for all $i>0$. Thus $N_0$ is flat over $B_0$.
\end{proof}

The next two lemmas complete the proof of statement (c).

\begin{lemma}
We have $W_n \subset Z_n$ for all $n$.
\end{lemma}

\begin{proof}
We must show that $M_n$ is flat over $X_n \setminus Z_n$. Thus let $x \in X_n \setminus Z_n$ be given. Choose $h \in A_n$ such that $X_n[1/h]$ is an open neighborhood of $x$ contained in $X_n \setminus Z_n$. We have a natural surjective map $\pi \colon X \to X_n$, under which $Z$ maps to $Z_n$. Since $X_n[1/h]$ does not meet $Z_n$, it follows that $\pi^{-1}(X_n[1/h])=X[1/h]$ does not meet $Z$, and so $M$ is flat over $X[1/h]$. Algebraically, this means that $M[1/h]$ is a flat $A[1/h]$-module. Taking $G(n)$-invariants and applying Lemma~\ref{lem:flat-invar}, we see that $(M_n)[1/h]$ is flat as an $(A_n)[1/h]$-module, which shows in particular that $M_n$ is flat at $x$, as required. Note that inverting $h$ commutes with forming of $G(n)$-invariants.
\end{proof}

\begin{lemma}
We have $W_n=Z_n$ for all $n \gg 0$.
\end{lemma}

\begin{proof}
Let $\fp \subset A$ be a $\GL$-stable prime that belongs to the non-flat locus $Z$. We have $M_{\fp}=\varinjlim (M_n)_{\fp_n}$ as a module over $A_{\fp} = \varinjlim (A_n)_{\fp_n}$, so if $(M_n)_{\fp_n}$ were flat infinitely often over $(A_n)_{\fp_n}$ then $M_{\fp}$ would be flat over $A_{\fp}$ \stacks{05UU}, which it is not. We thus see that there is some $n(\fp)$ such that $(M_n)_{\fp_n}$ is not flat over $(A_n)_{\fp_n}$ for all $n \ge n(\fp)$. Equivalently, this means that $\fp_n$ belongs to $W_n$ for all $n \ge n(\fp)$.

Let $Z^1, \ldots, Z^r$ be the irreducible components of $Z$, let $\fp^1, \ldots, \fp^r$ be their generic points, and let $n_0$ be the maximum of $n(\fp^1), \ldots, n(\fp^r)$. Then for $n \ge n_0$, we see that $W_n$ contains $\fp^1_n, \ldots, \fp^r_n$. Since these are dense in $Z_n$, it follows that $Z_n \subset W_n$. As we have already established the reverse containment, the result follows.
\end{proof}

\end{document}